\documentclass[12pt]{article}
\usepackage{amssymb}
\usepackage{amsmath}
\usepackage{amsthm}

\usepackage{hyperref}
\usepackage{url}
\usepackage{mathrsfs}
\usepackage{pdfpages}
\newtheorem{thm}{Theorem}[section]
\newtheorem{prop}[thm]{Proposition}
\newtheorem{lem}[thm]{Lemma}

\newtheorem{theor}[thm]{Theorem}
\newtheorem{cor}[thm]{Corolary}

\newcommand{\expect}{\mathbb{E}}
\newcommand{\R}{\mathbb{R}} %real line

\newcommand{\MG}{\mathcal{MG}} %set of Gibbs measures
\newcommand{\Gi}{\mathcal{G}}  %Gibbs structure
\newcommand{\Me}{\mathcal{M}}  %set of measures on given space
\newcommand{\bp}{\bar{\pi}} %Gibbs kernel
\newcommand{\pmax}{\pi^{+}} %maximal Gibbs kernel map
\newcommand{\pmin}{\pi^{-}} %minimal Gibbs kernel map
\newcommand{\bpmax}{\bar{\pi}^{+}} %maximal Gibbs kernel
\newcommand{\bpmin}{\bar{\pi}^{-}} %minimal Gibbs kernel
\newcommand{\pr}{\mathrm{pr}}
\newcommand{\Hom}{\mathrm{Hom}}
\newcommand{\mz}{mod $0$ }

\newcommand{\N}{\mathbb{N}}
\newcommand{\sA}{\mathscr{A}} %subalgebra A
\newcommand{\sB}{\mathscr{B}} %subalgebra B 
\newcommand{\sI}{\mathscr{I}} %subalgebra I 

\newcommand{\Sym}{\mathrm{Sym}}

\title{Random ordering formula for sofic and Rokhlin entropy of Gibbs measures}
\author{Andrei Alpeev\footnote{Chebyshev Laboratory, St. Petersburg State University, 14th Line V.O., 29B, Saint Petersburg 199178 Russia,  \href{mailto:a.alpeev@spbu.ru}{a.alpeev@spbu.ru}}   }

\begin{document}

\maketitle
\begin{abstract}
We prove the explicit formula for sofic and Rokhlin entropy of actions arising from some class of Gibbs measures. It provides a new set of examples with sofic entropy independent of sofic approximations. 
It is particularilly interresting, since in non-amenable case Rokhlin entropy was computed only in case of Bernoulli actions and for some examples with zero Rokhlin entropy. As an example we show that our formula holds for the supercritical Ising model. We also establish a criterion for uniqueness of Gibbs measure by means of $f$-invariant.
\end{abstract}

\section{Introduction}

In the paper \cite{B10a} Lewis Bowen defined a new invariant, so-called {\em sofic entropy} for measure preserving action of {\em sofic group}, which led to a great progress in the question of Bernoulli shifts isomorphism. The class of sofic groups is huge, containing all the amenable groups and all the residually finite groups. Currently it is not known whether all the groups are sofic or not.
This invariant depends on so-called {\em sofic approximation}, a sofic group can have numerous of them. 
In some cases sofic entropy was computed: for Bernoulli shifts (of finite base space entropy) it was done in the very paper \cite{B10a}. In the paper \cite{B11} Bowen have computed sofic entropy for a class of algebraic actions of residually finite groups with respect to the sofic approximations coming from their residually finite approximations. This result was later generalized by Hayes \cite{H14} to a class of algebraic action over all the sofic groups. In these results the very interresting phenomenon occurs: sofic entropy does not depend on the sofic approximation. Although, in the work \cite{C15} Carderi showed an example of the action having nonnegative sofic entropy for one sofic approximation and $-\infty$ for another (the latter actually corresponds to some kind of degeneration in the defintion). It is an interesting problem to clarify, whether action can have two nonnegative values of sofic entropy. 
In my work \cite{A15} I have proved that for the class of actions coming from Gibbs measures sofic entropy does not depend on the sofic approximation. 

Another entropy-type invariant the {\em Rokhlin entropy} was defined and investigated in works \cite{S14a}, \cite{S14b}, \cite{S16} by Seward, and in upcoming work \cite{AS17} by Seward and myself.
This entropy is defined as the infimum of Shannon entropies of all the generating partitions relative to the subalgebra of invariant subsets.
The name praises famous theorem by Rokhlin stating that Kolmogorov-Sinai entropy of aperiodic ergodic actions equals to the infimum of Shannon entropies of generating partitions. Rokhlin entropy seems to be extremely hard to compute for actions of nonamenable groups (for free actions of amenable groups it equals to the usual entropy). It is known that Rokhlin entropy provides the upper bound for sofic entropy (see \cite{B10a} and \cite{Ke13}, and \cite{AS17} for non-ergodic case). In fact, this provides essentially the only way to get a lower bound for the Rokhlin entropy. 

In the preprint \cite{S16} Seward gave a nontrivial upper bound for the Rokhlin entropy. In order to state it let us introduce some notation. Let $G$ be a countable group with fixed p.m.p. action on the standard probability space $X$. Let $(\xi_g)_{g \in G}$ be an i.i.d. random process such that each $\xi_g$ is uniformly distributed on the unit interval. Let $L_{\xi}$ denote the set of all such $g \in G$ that $\xi_g < \xi_e$ ($e$ stands for the group identity). For the partition $\alpha$ let us denote $\alpha^{L_{\xi}}$ the minimal subalgebra, with respect to which all the partitions $\alpha^g$ for $g \in L_{\xi}$ are measurable, where $\alpha^g$ denotes the $g$-shifted partition $\alpha$ (see section \ref{preliminaries} for the precise definitions). 
\begin{theor}[Seward, \cite{S16}]\label{Seward inequality}
Let $\alpha$ be a generating partition for a p.m.p. action of the countable group $G$. Then the Rokhlin entropy of this action is bounded from above by the quantity $\expect_{\xi}H(\alpha \vert \alpha^{L_{\xi}})$.
\end{theor}

In the work \cite{GS15}, Gaboriau and Seward provided an explicit upper and lower bounds for sofic and Rokhlin entropy of a class of algebraic actions coming from the celebrated Ornstein-Weiss construction \cite{OW87}.

Gibbs measures constitute a very important type of stochastic processes.

Let us define invariant Gibbs measures over groups. Let $G$ be a countable group and $A$ be a finite set of size at least 2. We can define shift-action of $G$ on $A^G$ by the formula $(gx)(h) = x(hg)$ for $x \in A^G$ and $g,h \in G$. A potential is any function $\varphi: A^G \to \R$ which depends only on the restriction of $x \in A^G$ to the finite subset $D \Subset G$. For $\Lambda$ a finite subset of $G$ we will set $\bp_{\Lambda}$ to be a probability kernel $\Me(A^G) \to \Me(A^G)$ such that for any $x \in A^G$ we will have that $\bp_{\Lambda}(\delta_x)$ is supported on the set of such $y \in A^G$ with $\pr_{G \setminus \Lambda}(x) = \pr_{G \setminus \Lambda}(y)$ ($\pr_{\Lambda}$ stands for the natural projection $A^G \to A^{\Lambda}$). We also require that $\bp_{\Lambda}(\delta_x)(\lbrace y \rbrace)$ for $y$ in the support is proportional to 
\[
e^{-\sum_{Dg \cap \Lambda \neq \varnothing} \varphi(gy)}.
\] 
Borel probability measure $\nu$ on $A^G$ is said to be {\em Gibbs measure} for the potential $\varphi$ if for any finite subset $\Lambda$ of $G$ we have that $\bp_{\Lambda}(\nu)=\nu$.

In this work I will first show the explicit lower bound for the sofic entropy of shift-actions on Gibbs measures satifying so-called Dobrushin condition. In fact, by the very similar argument one can prove directly the upper bound also and show that sofic entropy is the same for all the sofic approximations and can be computed by means of explicit formula. Instead we use the upper bound \ref{Seward inequality} by Seward for Rokhlin entropy, which happens to coincide with the lower bound. %Effectively it provides us with the explicit formula for the Rokhlin entropy of such actions. 
The third invariant, so-called {\em $f$-invariant}, was introduced by Bowen in the work \cite{B10a}. We will give the definition later. In the enquiry of sofic entropy for Gibbs measures we will hugely rely on the behaviour of Gibbs measures over sofic approximation. This is closely related to the study of Gibbs measures over sparse graphs, see \cite{DM10a}, \cite{DM10b}, \cite{AuP17}. 

\begin{theor}\label{main eq}
Suppose $\nu$ is a Gibbs measure for the potential $\varphi$, and $\alpha$ is a canonical alphabet generating partition. Suppose that one of the following holds:
\begin{enumerate}
\item Gibbs structure correspondent to the potential $\varphi$ satisfies the Dobrushin uniqueness condition;
\item Gibbs structure correspondent to the potential is attractive and posesses unique Gibbs measure.
\end{enumerate}
Then the sofic entropy of the correspondent shift-action is the same for all the sofic approximations. It also equals to the Rokhlin entropy and its value is given by the formula $\expect_{\xi}H(\alpha \vert \alpha^{L_{\xi}})$. If group $G$ is free then the latter also equals to the $f$-invariant.
\end{theor}
The proof will be given below.

The first condition could be satisfied by any finitely-determined function $\varphi$ after multiplication by a small enough coefficient. Example for the second one is Ising model posessing unique Gibbs measure. 

The proof will mainly rely on Seward's bound and two additional statements. To formulate them we will need the notion of {\em saturation}. Consider a $G$ shift-action on $A^G$ ($A$ is a finite set) with some invariant Borel probability measure. Let $\alpha$ be a canonical alphabet generating partition. If $\sA$ is a subalgebra, then we define its saturation $\widetilde{\sA}$ to be $\bigcap_{F} {\alpha^{G \setminus F} \vee \sA}$(the intersection is taken along all the finite subsets $F$ of $G$).

\begin{theor}\label{main ineq}
Suppose $\mu$ is a unique Gibbs measure for a potential $\varphi$, and $\alpha$ is a canonical alphabet generating partition. Then $h(\mu)\geq \mathbb{E}_{\xi}H(\alpha \vert \widetilde{\alpha^{L_{\xi}}} )$.
\end{theor}
It will be proved in the section \ref{gibbs over sofic} using the random ordering argument from \cite{BCKL13}. We note also that this type of argument was used to compute mean entropy for processes on graphs Benjamini-Schramm convergent to uniform trees in \cite{AuP17}.

In the following proposition we will refer to a slightly more general definition of Gibbs structure (see subsection \ref{gibbs measures general}) and. Definition of Dobrushin condition could be found in the susbection \ref{sec:dobrushin} and that of attractive Gibbs structure in \ref{sec:attractive}.
\begin{prop}\label{saturation lemma}
Suppose $\Gi$ is a Gibbs structure having unique Gibbs measure $\nu$.
Suppose that either of these condition holds:
\begin{itemize}
\item $\Gi$ satisfies Dobrushin condition,
\item $\Gi$ is attractive.
\end{itemize}
Then for any $S \subset G$ we have that $\widetilde{\alpha^{S}}$ is equivalent to $\alpha^S$ $\nu$-\mz.
\end{prop}
\begin{proof}
It follows from lemma \ref{saturation general}, lemma \ref{dobrushin good} and lemma \ref{attractive good}.
\end{proof}

\begin{proof}[Proof of theorem \ref{main eq}]
$h(\nu)$ will stand for sofic entropy and $h_{Rok}(\nu)$ for Rokhlin entropy. We know that $h(\nu) \leq h_{Rok}(\nu)$ (a well known inequality, see for example \cite{B10a}), $h_{Rok} \leq \expect_{\xi}H(\alpha \vert \alpha^{L_\xi})$ (by theorem \ref{Seward inequality}, we can apply it since action is essentially free by lemma \ref{gibbs free}), and $h(\nu) \geq \expect_{\xi}H(\alpha \vert \alpha^{L_\xi})$ (a combination of the previous proposition and theorem \ref{main ineq}). Together it obviously implies that $h_{Rok}(\mu)=h(\mu)=\expect_{\xi}H(\alpha \vert \alpha^{L_\xi})$. %Suppose now, that the group $G$ is free. 
For actions of free groups the proposition \ref{f-invariant and sofic} implies that the $f$-invariant could be expressed by means of the same formula.
\end{proof}

We note that these results could be instrumental for encountering phase transitions for Gibbs structures over free groups. It is said that {\em phase transition} occurs if there are more than one Gibbs measures.

\begin{theor}
Suppose $\nu$ is a shift-invariant Gibbs measure for some Gibbs structure over the free group. If corresponent action have nonpositive $f$-invariant, then the phase transition occurs.
\end{theor}
\begin{proof}
For unique Gibbs measure $\nu$ we will have by the theorem \ref{main ineq} and proposition \ref{f-invariant and sofic} that $f$-invariant is not smaller than $\expect_{\xi}H(\alpha \vert \widetilde{\alpha^{L_\xi}})$ which is positive. 
\end{proof}

These result were partially announced in \cite{A16}.

\subsection{Example: Ising model}
A standard Ising model over group $G$ with fixed finite generating set $S$ and  parameter $\beta$ is defined as follows. Alphabet $A=\lbrace -1, 1\rbrace$, potential $$\varphi(x)= - \beta\sum_{s \in S} x(e)x(s),$$ there  $\beta>0$. Using proposition \ref{attractive criterion} we can easily check that correspondent Gibbs structure is attractive. So the conclusion of the theorem \ref{main eq} holds whenever Gibbs measure is unique.  

Ising model can provide us with example of Gibbs measure with negative $f$-invariant. Let $G$ be a free group with a free generating set $S$.
Consider a standard Cayley tree for the free group (vertex set is group and two elements are connected whenever they have the form $(g,gs)$). We will construct the measure on $A^G$ by assigning the uniform distribution on $A$ for $e\in G$ first. Then we will step-by-step propagate this measure along the edges using the transition matrix
\[
\left( 
\begin{array}{cc}
 \frac{e^{2\beta}}{1+e^{2\beta}} & \frac{1}{1+e^{2\beta}} \\
\frac{1}{1+e^{2\beta}} & \frac{e^{2\beta}}{1+e^{2\beta}}  
\end{array}
\right).
\] 
We refer to \cite{B10d} for details on Markov processes on trees.
It is not hard to check that this measure will be an invariant Gibbs measure for Ising model. Since our measure is given by Markov chain, by the result of \cite{B10d} we can compute the $f$-invariant
\[
f_{\beta,m}=(1- m )\log 2 -  m \left(\frac{e^{2\beta}}{1+e^{2\beta}} + \frac{1}{1+e^{2\beta}}\right),
\] 
there $m =\lvert S \rvert$.
If $G$ is non-abelian, for big enough $\beta$ we will have that this expression is non-positive and hence, phase transition occurs.
So we have the following 
\begin{cor}
If $f_{\beta,m} \leq 0$ then Gibbs measure for Ising model with parameter $\beta$ over the free group of order $m$ has more than one Gibbs measure. If Gibbs measure for this model is unique then its sofic entropy, Rokhlin entropy and $f$-invariant are all equal to $f_{\beta,m}$.
\end{cor}

{\em Acknowledgements}

I would like to thank Miklos Abert, Brandon Seward and my advisor Anatoly Vershik for fruitful conversations.
Research is supported by the Russian Science Foundation grant №14-21-00035. 

%\subsection{Structure of the paper}

%In section \ref{preliminaries} we will introduce notation and some basic notions.
%In subsection \ref{sofic entropy} we remind the definition of sofic groups and sofic entropy. After, in subsection \ref{f-invariant} the definition of $f$-invariant is given. 
%Section \ref{gibbs measures general} contains general discussion of Gibbs measures: standard definitions and facts, including definition of Gibbs structure, Gibbs measure, formulation of Dobrushin uniqueness condition. It also contains proof of lemma \ref{saturation lemma}, which, after necessarily definitions from section \ref{gibbs over sofic} implies the proposition \ref{saturation lemma}. After that, in section \ref{gibbs over sofic} we consider the particular case of Gibbs structures over groups, and especially, over sofic groups. We can mimic then Gibbs structures on sofic approximations, which provides us with the tool to compute sofic entropy in case of unique Gibbs measure potentials. Then we use the random ordering argument to prove theorem \ref{main ineq}.

\section{Preliminaries and conventions}\label{preliminaries}
Symbol $\Subset$ will stand for "finite subset".

All the measures in the sequel will be Borel probabilistic. For the standard Borel space $X$ we will denote $\Me(X)$ the set of all such measures. If $X$ is a metrizable compact we will endow $\Me(X)$ with the weak* topology (unless otherwise is specifically noted). For $X$,$Y$ two standard Borel spaces, $f: X \to Y$ a Borel map and $\mu \in \Me(X)$ we will slightly abuse the notation by writing $f(\mu)$ for the push-forward of measure $\mu$. For two Borel probability measures $\mu_1$ and $\mu_2$ on standard Borel spaces $X_1$ and $X_2$ respectively we will denote $\mu_1 \otimes \mu_2$ the product measure on $X_1 \times X_2$.

Let $X$ be a standard Borel space. We will sometimes use "subalgebra" instead of $\sigma$-subalgebra.
Let $\mu$ be a measure on $X$. For  $\sigma$-subalgebra $\sA$ and $x \in X$ we will denote $\mu\vert_x^{\mu}$ --- the measure correspondent to the point $x$ according to the decomposition of $\mu$ along the subalgebra $\sA$. We refer the reader to \cite{EW11}, section 5.3. for more details. If $\sA \subset \sB$ are two countably-generated subalgebras, then $\mu\vert_x^{\sA}\vert_x^{\sB}=\mu\vert_x^{\sB}$ for $\mu$-a.e. $x \in X$.
Two subalgebras $\sA$ and $\sB$ on $X$ are said to be $\nu$-\mz equivalent if for any $B \in \sA$ there is $B' \in \sB$ with $\nu(B \Delta B')=0$ and vice versa. Two subalgebras $\sA$,$\sB$ are $\mu$ - \mz equivalent iff $\mu\vert_x^{\sA}=\mu\vert_x^{\sB}$ for $\mu$-a.e. $x \in X$. We will denote $\sA \vee \sB$ the minimal subalgebra, containing $\sA$ and $\sB$.
%Let $X$ be a standard Borel space. %For $\sA$ a subalgebra and $\mu$ a measure on $X$ its $\mu$-completion is a
Let $\sA,\sB$ be two subalgebras. If the measure $\mu$ on $X$ is specified then by the intersection of this two subalgebras we will usually mean the intersection of their $\mu$-completions.

Let $X$ be a standard probaility space or a standard Borel space. {\em Partition} is a countable or finite collection of its disjoint subsets whose union is the whole space. It makes sense to consider partition as a particular case of subalgebra.
The {\em Shannon entropy} of a partition $\alpha$ of the satandard probability space $(X,\mu)$ is defined by the formula
\[
H(\alpha)=-\sum_{B \in \alpha} \mu(B) \log(\mu(B))
\]
with the usual convention $0\log 0 =0$. If $\mu$ is the measure on finite or countable set, we will define its entropy $H(\mu)$ as entropy of the partition whose elements are the atoms of this measure.

For two partitions $\alpha$, $\beta$ we will define $\alpha \vee \beta = \lbrace A \cap B \vert A \in \alpha, B \in \beta, A \cap B \neq \varnothing \rbrace$. We will denote
\[
H(\alpha  \vert \beta) = H(\alpha \vee \beta) - H(\beta)
\]
the relative Shannon entropy.
For $\alpha$ a partition and $\sA$ a subalgebra 
\[
H(\alpha\vert \sA) = \inf_{\beta} H(\alpha \vert \beta )
\]
(the infimum is taken along all the $\sA$-measurable partitions $\beta$ of finite Shannon entropy) will stand for the relative Shannon entropy with respect to a subalgebra. It is known that the result will not change if we will take the infimum along all the finite $\sA$-measurable partitions. 

Let $( \sA_i )_{i \in \N}$ be a \mz increasing sequence of subalgebras and let $\sB = \bigcup_{i \in \N} \sA_i$. It is well known that for any partition $\alpha$ we have
\[
H(\alpha \vert \sB) = \lim_{i \to \infty} H(\alpha \vert \sA_i).
\]
For a \mz decreasing sequence $( \sA_i )_{i \in \N}$ of subalgebras let $\sB = \bigcap_{i \in N} \sA_i$. For a finite partition $\alpha$ we have 
\[
H(\alpha\vert\sB)=\lim_{i \to \infty} H(\alpha \vert \sA_i).
\]

For a partition $\alpha$ of finite entropy and a countably-generated subalgebra $\sA$ we have.
\[
H(\alpha \vert \sA) = \int_{X} H_{\mu\vert_x^{\sA}}(\alpha) \; d\mu(x).
\]

Any function $Y$ on the standard probability space $\Omega$ with finite or countable range defines naturally a partition of $\Omega$. So we will use all the notation above for such functions (random variables) as well. For random variables $Y'$ and $(Y_i)_{i \in \N}$ we will denote
\[
H(Y' \vert (Y_i)_{i \in G})
\] 
the relative entropy of $Y'$ with respect to subalgebra generated by functions $(Y_i)_{i \in \N}$.

Let $X$ and $Y$ be two metric compacta. A {\em continuous probability kernel} is an affine Borel map $\bp: \Me(X) \to \Me(Y)$. It is completely determined by its values on $\delta$-measures, moreover for any Borel map $\pi: X \to \Me(Y)$ there is a unique probability kernel $\bp$ such that $\bp(\delta_x)=\pi(x)$. %We will call function $\pi$ for kernel $\bp$ a {\em base function}.
For any measure $\mu$ on $X$ and any function $\varphi \in C(Y)$ we have 
\[
\int_Y \varphi(y) d(\bp(\mu))(y) = \int_X d\mu(x) \int_Y \varphi(y) d(\pi(x))(y).
\]
If $\pi$ is a Borel map from $X$ to $\Me(X)$, then the equation above clearly defines a map $\Me(X) \to \Me(Y)$. We will call a map which could be obtained in such a way a {\em probability kernel}. Clearly, continuous probability kernel is a special case of probability kernel. 

Let $(\pi_n)$ be a sequence of maps $X \to \Me(Y)$ and $\pi' : X \to \Me(Y) $. Let $(\bp'_n)$, $\bp'$ be correspondent probability kernels. It is not hard do check that $(\bp_n)$ converges to $(\bp')$ pointwise iff $(\pi_n)$ converges pointwise to $\pi'$.

Let $V$ be a set and let $\lbrace A_v\rbrace_{v \in V}$ be a collection of sets indexed by elements of $V$. We will denote $\pr_{\Lambda}: \prod_{v \in V}A_v \to \prod_{v \in \Lambda}A_v$ the natural projection for $\Lambda \subset V $. 
 
\begin{lem}\label{special kernel decomp}
Suppose $\bp: \Me(X \times Y) \to \Me(X \times Y)$ is such a probability kernel that for any $(x,y) \in X \times Y$ we have $\pr_X(\bp(\delta_{(x,y)}))=\delta_x$,  and $\bp(\delta_{(x,y_1)})=\bp(\delta_{(x,y_2)})$ for $x \in X$ and $y_1,y_2 \in Y$ ($\pr_X$ stands for the natural projection $X \times Y \to X$ ). 
Then for $\mu \in \Me(X \times Y)$ we have $\bp(\mu)=\mu$ iff $\mu\vert_{(x,y)}^{\sA} = \bp(\delta_{(x,y)})$ for $\mu$-a.e. $(x,y)$ in $X \times Y$. 
\end{lem} 

A {\em dynamical system} or a {\em p.m.p. action} or a {\em $G$-space} is a measure-preserving action of a countable group on the standard proobability space. Let us fix a p.m.p. action of countable group $G$ on the standard probability space $(X,\mu)$. For a partition $\alpha$ and $g \in G$ we will denote $\alpha^g = \lbrace g^{-1}(B) \vert B \in \alpha\rbrace$. For $F \subset G$. Let $\alpha^F = \bigvee_{g \in F} \alpha^g$. It makes sense to consider the latter as a partition for finite $F$ and as a subalgebra in the other case.
We will say that partition $\alpha$ is {\em generating} if $\alpha^G$ $\mu$-\mz equivalent to the algebra of all the measurable sets. It is known that $\alpha$ is a generating partition iff there is such a conull subset $X'$ of $X$ that $x,y \in X'$ are non-equal whenever there is such $g \in G$ that $g(x)$ and $g(y)$ belong to different elements of $\alpha$. 

{\em Rokhlin entropy} is the infimum of Shannon entropies of generating partitions relative to the subalgebra of invariant subsets:
\[
h_{Rok}=\inf \lbrace H(\alpha \vert \sI)\text{, $\alpha$ is generating partition}\rbrace,
\]
$\sI$ stands for the subalgebra of invariant sets.
For ergodic action this definition reduces to mere infimum of entropies along all the generating partitions.

Let $G$ be a countable group and $A$ be a finite set (an {\em alphabet}). On the space $A^G$ endowed with the product topology (assuming the discrete topology on $A$) we will define the {\em shift action} by the formula
\[
(gx)(h)= x(hg)
\]
for $x \in A^G$ and $g,h \in G$. This action is continuous. Measure $\nu$ on $A^G$ is said to be invariant if $g(\nu)=\nu$ for every $g \in G$. Let $B_a$ for $a \in A$ be the set of such $x \in A^G$, that $x(e)=a$ ($e$ stands for group identity). Partition $\alpha=\lbrace B_a\vert a \in A\rbrace$ will be called a canonical alphabet partition. For $\sA$ a subalgebra we will denote 
\[
\widetilde{\sA}=\bigcap_{F \Subset G} \sA \vee \alpha^{G \setminus F},
\]
we will call $\widetilde{\sA}$ a {\em saturation} of $\sA$.

\subsection{Sofic groups and sofic entropy}\label{sofic entropy}

For a finite set $R$ we will denote $\Sym(R)$ the group of its permutations. We define a normalized Hamming distance $d_H$ on $\Sym(R)$ by 
\[
d_H(g_1,g_2)=\frac{\left\lvert\lbrace r \in R, g_1(r) \neq g_2(r)\rbrace \right\rvert}{\lvert R\rvert}.
\]

Let $G$ be a countable group. The {\em sofic approximation} is defined by the sequence of finite sets $(V_i)_{i \in \N}$ and the sequence of maps $( \sigma_i )_{i \in \N}$, $\sigma_i^{\cdot} : G \to \Sym(V_i)$, such that 
\begin{enumerate}
\item for $g_1 \neq g_2$ from $G$ we have $\lim_{i \to \infty} d_H(\sigma_i^{g_1},\sigma_i^{g_2}) = 1$,
\item for $g_1,g_2$ from $G$ we have 
$\lim_{i \to \infty} d_H(\sigma_i^{g_1} \circ \sigma_i^{g_2}, \sigma_i^{g_1 g_2})=0$.
\end{enumerate}

A group is called {\em sofic group} if it posesses at least one sofic approximation. From now on let $G$ be a sofic group and its sofic approximation be fixed.

We will say that an element $v \in V_i$ if $S$-good for $S \Subset G$ if
\begin{enumerate}
\item $\sigma_i^{g_1}(v) \neq \sigma_i^{g_2}(v)$ for $g_1 \neq g_2$ from $S$,
\item $(\sigma_i^{g_1} \circ \sigma_i^{g_2})(w)=\sigma_i^{g_1 g_2}(w)$ for $g_1,g_2 \in S$ and $w \in \sigma_i^{S}(v)$,
\item $(\sigma_i^{g^{-1}}\circ \sigma_i^g)(w)=w$ for $w \in \sigma_i^S$,
\item If $w \in \sigma_i^{S}(v)$, $t \in V_i$ and $g \in S$ are such that $w=\sigma_i^g(t)$, then $t=\sigma_i^{g^{-1}}(w)$.
\end{enumerate}

A simple counting argument enables one to prove the folowing lemma:
\begin{lem}\label{sofic is good}
Let $S$ be any finite subset of $G$. Let $V'_i$ (for every $i \in \N$) be a set of $S$-good points in $V_i$. Then $\lim_{i \to \infty} \lvert V_i'\rvert / \lvert V_i \rvert = 1$.
\end{lem}

Let $A$ be a finite set. We define a collection of maps $\theta_v: A^{V_i} \to A^G$ by the formula $\left(\theta_v(\tau)\right)(g)=\tau\left(\sigma_i^g(v)\right)$. 
Let $\nu$ be a shift-invariant Borel probability measure on $A^G$. Let $l$ be some metric for the weak* topology on $\Me(A^G)$.
For $\varepsilon>0$ and $i \in \N$ we denote $\Hom(i,\varepsilon)$ the set of all such $\tau \in A^{V_i}$ that 
\[
l\left(\frac{1}{\lvert V_i\rvert}\sum_{v \in V_i}\delta_{\theta_v(t)}, \nu \right) < \varepsilon.
\]
Then {\em sofic entropy} of shift-action endowed with  measure $\nu$ is defined as 
\[
h(\nu)=\inf_{\varepsilon>0} \limsup_{i \to \infty} \frac{\log \lvert \Hom(i,\varepsilon)\rvert}{\lvert V_i\rvert}.
\]
It was originally defined in the work \cite{B10b}. It seemingly depends on the choice of the symbolic representation. In fact as was proved by Bowen this quantity is measure conjugacy invariant. Up to date survey on sofic entropy could be found in \cite{B17}.

\subsection{$f$-invariant}\label{f-invariant}

Let $G$ be a finitely generated free group with the set of generators $s_1,\ldots,s_m$. Let us fix a p.m.p. action of this group on the standard probability space. Suppose it has a finite generating partition $\alpha$. Let $C_r$ be the $r$-ball around the group identity with respect to the word-metric generated by the set $\lbrace s_1^{\pm}, \ldots, s_m^{\pm} \rbrace$. 
We will define {\em $f$-invariant} of given $G$-action as
\[
h_f = \inf_{r \in \N} \left\lbrace (1-2m)H(\alpha^{C_r}) - \sum_{i=1}^m H(\alpha^{C_r \cup s_i C_r}) \right\rbrace.
\]

In the paper \cite{B10a} Bowen defined this quantity and proved that it does not depend on the choise of the finite generating partition. He also showed that it equals to the base space entropy of Bernoulli shifts thus solving the isomorphism problem for Bernoulli actions of free groups. 

In the paper \cite{H14} Hayes proved, building upon the result of \cite{B10c}, the result relating sofic entropy  and $f$-invariant.

\begin{prop}\label{f-invariant and sofic}
Let us fix an essentially free action of free group. Then its $f$-invariant is bounded between supremum and infimum of its sofic entropies along all the sofic approximations.
\end{prop}

\section{Gibbs measures}
\subsection{General definitions}\label{gibbs measures general}
We define a {\em Gibbs structure} $\Gi$ as a tripple of a finite or countable vertex set $V$, a collection of finite sets({\em alphabets}) $( A_v)_{v \in V}$, and a potential $( \Phi_T )_{T \Subset V}$: a set of functions $\Phi_T: \prod_{v \in T} A_v \to \R$ such that for any $v \in V$ all but finitely many $\Phi_T$ for $T \ni v$ are identically zero. We wil also apply this functions not only to respective $\prod_{v \in T} A_v$, but to $\prod_{v \in V} A_v$ also. Consider the set $\Omega= \prod_{v \in V} A_v$ endowed with the product topology inherited from the discrete topologies on $A_v$'s. For every $W \subset V$ we define a natural projection map $\pr_{V}: \prod_{v \in V}A_v \to \prod_{w \in W}A_w$. We will denote $\sB(W)$ preimage of natural Borel algebra on $\prod_{w \in W} A_w$ under the map $\pr_{W}$. 
Now we would like to define the set of Gibbs measures for this Gibbs structure. At first, let us define the collection of maps $( \pi_{\Gi,\Lambda} )_{\Lambda \Subset G}$, $\pi_{\Gi,\Lambda}$ will map the point $\omega \in \Omega$ to the probability measure $\pi_{\Gi,\Lambda}(\omega)$ supported on the set of points $\omega'$ such that $\pr_{V \setminus \Lambda}(\omega)=\pr_{V \setminus \Lambda}(\omega')$, and 
 such that the probability of the point $\omega'$ from the support is proportional to $\exp{(-\sum_{T \cap \Lambda \neq \varnothing} \Phi_T(\omega'))}$. Note that this uniquelly determines maps and they are continuous. %(there the sum is taken over all the finite subsets $T$ of $V$ having nonempty intersection with $\Lambda$).
 We then expand each of this maps to probabilistic kernel $\bp_{\Gi,\Lambda}: \Me(\Omega) \to \Me(\Omega)$, $\Lambda \Subset V$. % It is clear that these kernels would be weak-$\star$ continuous.
 It could be easily checked that for $\Lambda \subset \Lambda' \Subset V$ we have $\bp_{\Gi,\Lambda} \circ \bp_{\Gi,\Lambda'} = \bp_{\Gi,\Lambda'} \circ \bp_{\Gi,\Lambda} = \bp_{\Gi,\Lambda'}$. The set of {\em Gibbs meaures} $\MG$ for given Gibbs structure $\Gi$ is defined to be the set of all such Borel probability measures $\nu$ on $\Omega$ that $\bp_{\Gi, \Lambda}(\nu)=\nu$ for all $\Lambda \Subset V$. It follows from compactness that this set is nonempty. It is not hard to see that in the special case of finite vertex set the set of Gibbs measures always contains unique measure which is defined by the fact that measure of any $\omega \in \Omega$ is propotional to $\exp(\sum_{T \subset V}\Phi_T(\omega))$.

The following lemma directly follows from lemma \ref{special kernel decomp}:
\begin{lem}
For measure $\nu$ we have $\bp_{\Gi,\Lambda}(\nu)=\nu$ iff for $\nu$-a.e. $\omega \in \Omega$ we have $\nu \vert_{\omega}^{\sB(V \setminus \Lambda )} = \pi_{\Gi,\Lambda}(\omega)$.
\end{lem}

For $\omega \in \Omega$ and $S \subset V$ we denote $\Gi_{S,\omega}$ new Gibbs structure having the only difference with $\Gi$ in the collection of alphabets: $A_v'=A_v$ for $v \notin S$ and $A_v'=\lbrace \omega_v\rbrace$ for $v \in S$. It is worth to note that Gibbs measures for this Gibbs structures are measures on $\Omega=\prod_{v \in V}A_v$. 

Following lemma is obvious.
\begin{lem}
Measure $\nu$ is Gibbs measure for $\Gi_{S,\omega}$ iff $t(v)=\omega(v)$ for $\nu$-a.e. $t \in \Omega$ and for every $v \in S$, and $\bp_{\Gi,\Lambda}(\nu)=\nu$ for every $\Lambda \Subset V \setminus S$.
\end{lem}

We now show that Gibbs measure for $\Gi$ could be decomposed into Gibbs measures for $\Gi_{F,\omega}$.

\begin{lem}\label{gibbs decomp}
Suppose $\nu \in \MG$, $F \subset G$. Then $\nu\vert_{\omega}^{\sB(F)} \in \MG_{F,\omega}$ for $\nu$ -- a.e. $\omega \in \Omega$.
\end{lem}
\begin{proof}
It is easy to see that for any $v \in F$ and $\nu$-a.e. $\omega\in \Omega$ we have $\nu(\lbrace t\in\Omega\vert t(v)=\omega(v)\rbrace)=1$. For $\Lambda \Subset V \setminus F$ we denote $\sA=\sB(V \setminus \Lambda)$. Since $\sA \subset \sB(F)$ and both are countably generated, we have $\nu\vert_{\omega}^{\sB(F)}\vert_{\omega}^{\sA}=\nu\vert_{\omega}^{\sA}=\pi_{\Gi,\Lambda}$ for $\nu$-a.e. $\omega \in \Omega$ (we used the fact that $\nu \in \MG$). Putting this together we get the desirable result.
\end{proof}

\begin{lem}
Suppose is $( F_k )_{k \in \N}$ is such a sequence of the subsets of $G$, that $F_i \subset F_j$ for $i > j$. Let $F=\bigcap_{i \in \N}F_i$. Suppose $\omega \in \Omega$ is some point and $( \nu_k )_{k \in \N}$ is a sequence of measures weakly converging on $\Omega$ such that $\nu_k \in \MG_{F_k,\omega}$ for $k \in \N$. Then the limiting measure belongs to $\MG_{F,\omega}$.
\end{lem}
\begin{proof}
Let us denote $\nu$ the limiting measure.
It easy to see that $\nu(\lbrace t \in \Omega \vert t(v)=\omega(v)\rbrace)=1$ for $v \in F$, since the latter condition is closed and is satisfied for $\nu_k$, $k\in\N$. We also note that $\bp_{\Gi,\Lambda}(\nu)=\nu$ for every $\Lambda \Subset V \setminus F$ by the same reason.  
\end{proof}

In order to state the {\em Markov property} we will define the boundary. Let $\Lambda$ be a finite subset of $V$. Its boundary $\partial \Lambda$ is defined as $\bigcup \lbrace T \setminus \Lambda \;\vert T : T \cap \Lambda \neq \varnothing, \Phi(T) \not\equiv 0\rbrace$, where the union is taken over all such $T \Subset V$ that $\Phi_T$ is not identically zero. Obviously, $\partial \Lambda$ is finite. If there could be ambiguity we will signify the Gibbs structure: $\partial_{\Gi}\Lambda$.

\begin{lem}
For any $\Lambda \Subset V$ we have $\pi_{\Gi,\Lambda}$ is $\sB(\partial \Lambda)$-measurable.
\end{lem}

So, for $\Lambda \Subset V$ we can naturally define $\pi_{\Gi,\Lambda}: \prod_{v \in W}A_v \to \Me(\prod_{v \in W}A_v)$ for any $W \subset V$ suсh that $W \supset \Lambda \cup \partial \Lambda$. Also we can define $\bp_{\Gi,\Lambda}$ on $\Me(\prod_{v \in W}A_v)$.
Another (and more useful) way to formulate Mfrkov property is the following.
\begin{lem}\label{Markov}
Let $(X_v)_{v \in V}$ be a collection of random varibles, distributed according to some Gibbs measure for $\Gi$. Then for any $\Lambda \Subset V$ we have that the distribution of $(X_v)_{v \in \Lambda}$ relative to $(X_v)_{v \in V \setminus \Lambda}$ is \mz the same as the distribution of $(X_v)_{v \in \Lambda}$ relative to $(X_v)_{v \in \partial \Lambda}$.
\end{lem}

Let $W$ be a subset of $V$.
We will say that the probability measure $\nu$ on $\prod_{v \in W}A_v$ is $(\Gi,\Lambda)$-admissible if $\bp_{\Gi,\Lambda}(\nu)=\nu$. 
The fact that measure $\nu$ is Gibbs measure is locally testable.
\begin{lem}
Measure $\nu$ is Gibbs measure for Gibbs structure $\Gi$ iff for every pair $\Lambda, W \Subset V$ with $\Lambda \cup \partial \Lambda \subset W$ we have $\bp_{\Gi,\Lambda}(\pr_{W}(\nu))=\pr_{W}(\nu)$.
\end{lem}

\begin{lem}\label{approximation of gibbs measure}
Suppose Gibbs structure $\Gi$ has unique Gibbs measure $\nu$.
Suppose $S \Subset V$. Let us consider any metric compatiable with weak* topology on $\Me(\prod_{v \in S}A_v)$. Then for any $\varepsilon>0$ we can find such a pair $W,\Lambda \Subset V$, $\Lambda \cup \partial \Lambda \subset W$, that for every $(\Gi,\Lambda)$-admissible measure $\mu$ on $\prod_{v \in W}A_v$ we have $\pr_{S}(\mu)$ is $\varepsilon$-close to $\pr_S(\nu)$.
\end{lem}
\begin{proof}
Suppose the contrary. Let us pick a growing series $( \Lambda_i )_{i \in \N}$ and $( W_i )_{i \in \N}$ of subsets of $V$ with $\Lambda_i \partial_{\Gi}\Lambda \subset W$,  such that $S$-projections of $(\Gi,\Lambda_i)$-admissible measures $\mu_i$($\mu_i \in \Me(\prod_{w \in W_i}A_i)$) would not be $\varepsilon$-close to $S$-projection of $\nu$. Then we refine a subsequence such that its $R$-projections for every finite subsets $R$ are converging, it is easy to see that this projections would be compatiable and so, they will define a Gibbs measure on $\Omega$. But its $S$-projection would be $\varepsilon$-apart from $\pr(\nu)$, a contradiction.
\end{proof}

For subalgebra $\sA$ on $\Omega$ let us define its saturation to be $$\widetilde{\sA}=\bigcap_{F \Subset V}\sA \vee \sB(V \setminus F).$$

%In the following lemma we give a sufficient condition for some subalgebras not  \ref{saturation lemma}.

\begin{lem}\label{saturation general}
Let $\Gi$ be a Gibbs structure and $\nu$ --- its Gibbs measure. Let $S$ be a subset of $V$. Suppose that for $\nu$ - almost every $\omega \in \Omega$ we have that $\Gi_{S,\omega}$ has unique Gibbs measure.  
Then subalgebras $\sB(S)$ and $\widetilde{\sB(S)}$ are $\nu$-\mz equivalent. 
\end{lem}
\begin{proof}
Let $\sA=\sB(S)$.
Let us fix some vanishing nested  sequence of co-finite sets $\lbrace F_i\rbrace_{i \in \N}$, $F_i \subset F_j$ for $i> j$, $\bigcap_{i \in \N}F_i = \varnothing$, $V\setminus F_i \Subset V$ for $i \in \N$. It is easy to see that $\widetilde{\sA}=\bigcap_{i \in \N}\sA \vee \sB(F_i) = \bigcap_{i \in \N} \sB(S \cup F_i)$. Let $\sA_i = \sB(S \cup F_i)$.
For $\nu$-a.e. $\omega$ we have the following: $\nu \vert_{\omega}^{\sA_i} \to \nu \vert_{\omega}^{\widetilde{\sA}}$(by the measure-valued martingale convergence theorem, see \cite{EW11}, corollary 5.21, p. 144); for all $i \in \N $ and for a.e. $\omega \in \Omega$ holds $\nu\vert_{\omega}^{\sA_i} \in \MG_{S \cup F_i, \omega}$ and hence, we have that 
$\nu\vert_{\omega}^{\widetilde{\sA}} \in \MG_{S,\omega}$. On the other hand, we have that $\nu\vert_{\omega}^{\sA} \in \MG_{S,\omega}$ for $\nu$--- a.e. $\omega \in \Omega$. Since for $\nu$ --- a.e. $\MG_{S,\omega}$ is a one-point set, we have that $\nu\vert_{\omega}^{\widetilde{\sA}}=\nu\vert_{\omega}^{\sA}$ for $\nu$ - a.e. $\omega \in \Omega$. This implies that $\sA = \widetilde{\sA}$ modulo $\nu$. 
%and 
%$\nu\vert_{\omega}^{\sA} \in \MG_S$. So It is now easy to see that for almost every $\omega \in \Omega$ we have $\nu\vert_{\omega}^{\sA_i} \to \nu\vert_{\omega}^{\sA}$, so for $\nu$-a.e. $\omega \in \Omega$ we have $\nu\vert_{\omega}^{\sA}=\nu\vert_{\omega}^{\widetilde{\sA}}$. This implies that $\sigma$-subalgebras $\sA$ and $\widetilde{\sA}$ are $\mod-0$ equivalent. 
\end{proof}

\subsection{Dobrushin uniqueness condition}\label{sec:dobrushin}
To state Dobrushin uniqueness condition we will need some notation. At first let us remind that the {\em total variation} norm $\lVert \mu \rVert$ of measure (not necessarily probabilistic) $\mu$ is defined to be $\sup \lbrace \mu(B) - \mu(C) \rbrace$ (where the supremum is taken along all the pairs $B,C$ of measurable subsets).
Let 
\[
b_{v,u}=\sup_{\omega_1,\omega_2}\lVert pr_{\lbrace v\rbrace}(\pi_{\Gi,\lbrace v \rbrace})(\omega_1) - pr_{\lbrace v\rbrace}(\pi_{\Gi,\lbrace v \rbrace})(\omega_2)\rVert,
\]
where the supremum is taken along all the pairs $\omega_1,\omega_2 \in \Omega$ such that $\omega_1(w)=\omega_2(w)$ for all $w\neq v$, $w\in V$. Then let $$b_v=\sum_{u \neq v}b_{v,u}.$$ 
And finally let 
\[
b^{\star}=\sup_{v \in V} b_v.
\] 
We will say that Gibbs structure $\Gi$ satisfies {\em Dobrushin condition} if $b^{\star} < 1$.

\begin{theor}[Dobrushin, \cite{D68}]
If $\Gi$ satisfies Dobrushin condition then $\MG$ contains unique measure.
\end{theor}

The following could be proved by easy check.
\begin{lem}\label{dobrushin good}
Suppose $\Gi$ is a Gibbs structure satisfying  Dobrushin condition, $\omega \in \Omega$ and $F$ is some subset of $G$. Then $\Gi_{F, \omega}$ satisfies  Dobrushin condition also and hence have unique Gibbs measure .
\end{lem}

%\begin{lem}\label{saturation general}
%Let $\Gi$ be a Gibbs structure satisfying Dobrushin condition and $\nu$ --- its unique measure. Then for every $S \subset V$ we have that subalgebra $\sB(S)$ and $\widetilde{\sB{S}}$ are $\nu$-equivalent 
%\end{lem}
%\begin{proof}
%Let $\sA=\sB(S)$.
%Let us fix some vanishing nested  sequence of co-finite sets $\lbrace F_i\rbrace_{i \in \N}$, $F_i \subset F_j$ for $i> j$, $\bigcap_{i \in \N}F_i$, $V\setminus F_i \Subset V$ for $i \in \N$. It is easy to see that $\widetilde{\sA}=\bigcap_{i \in \N}\sA \vee \sB(F_i) = \bigcap_{i \in \N} \sB(S \cup F_i)$. Let $\sA_i = \sB(S \cup F_i)$.
%For $\nu$-a.e. $\omega$ we have the following: $\nu \vert_{\omega}^{\sA_i} \to \nu \vert_{\omega}^{\widetilde{\sA}}$(by the measure-valued martingale convergence theorem, see \cite{EW11}, corollary 5.21, p. 144); for all $i \in \N $ holds $\nu\vert_{\omega}^{\sA_i} \in \MG_{S \cup F_i, \omega}$ and 
%$\nu\vert_{\omega}^{\sA} \in \MG_S$. So It is now easy to see that for almost every $\omega \in \Omega$ we have $\nu\vert_{\omega}^{\sA_i} \to \nu\vert_{\omega}^{\sA}$, so for $\nu$-a.e. $\omega \in \Omega$ we have $\nu\vert_{\omega}^{\sA}=\nu\vert_{\omega}^{\widetilde{\sA}}$. This implies that $\sigma$-subalgebras $\sA$ and $\widetilde{\sA}$ are $\mod-0$ equivalent. 
%\end{proof}

\subsection{Attractive Gibbs structures}\label{sec:attractive}

In order to manipulate with attractive Gibbs structures we will need to establish some results and notation concerning stochastic domination. 

Suppose $X$ is a Polish space endowed with fixed order $\preceq$ on it. Suppose $\mu_1, \mu_2 \in \Me(X)$. We will say that measure $\mu_2$ stochastically dominates measure $\mu_1$ if there is a coupling of $\mu_1$ and $\mu_2$ whose support is a subset of $\preceq$, that is there is a measure $\mu'$ on $X \times X$ such that $\mu'(\preceq) = 1$ and $\pr_i(\mu')=\mu_i$, $i=1,2$ ($\pr_i$ mean a projection of $X \times X$ on the correspondent coordinate ). It is easy to see that if $X$ is a compact metriable space and $\preceq$ is a closed subsed of the cartesian product then the correspondent dominance order on $\Me(X)$ is a closed subset of $\Me(X) \times \Me(X)$. It is known that $\mu_1=\mu_2$ whenever $\mu_1 \preceq \mu_2$ and $\mu_2 \preceq \mu_1$.
  
Suppose for each $v \in V$ we have a partial order $\preceq$ on $A_v$. We then can establish an order on the whole space $\Omega$ (so-called Fortuin–Kasteleyn–Ginibre, FKG order) in the following way: $\omega' \preceq \omega''$ iff $\omega_v' \preceq \omega_v''$ for every $v \in V$. We will now establish a stochastic dominance order on $\Me(\Omega)$(and we will use the same sign $\preceq$).Gibbs structure $\Gi$ is said to be {\em attractive} if for every $\Lambda \Subset V$ the map $\bp_{\Gi,\Lambda}$ is monotone. It is clearly equivalent to $\pi_{\Gi, \Lambda}$ being monote for every $\Lambda \Subset V$.  

It is a well-known that we can require much less to establish attractiveness of Gibbs structures. 
\begin{prop}\label{attractive criterion}
Gibbs structure $\Gi$ is attractive if for every $v \in V$ we have that the map $\pi_{\Gi,\lbrace v \rbrace}$ is monotone.
\end{prop}

Suppose that for every $v \in V$ order on $A_v$ have unique maximal element and unique minimal element. It will imply that induced order on $\Omega$ has unique maximal and unique minimal element also, and we will denote them $\omega^+$ and $\omega^-$ respectively. 

We will now say that the Gibbs structure $\Gi$ is {\em attractive} if for every $\Lambda \Subset V$. The map $\bp_{\Gi,\Lambda}$ is monotone. The latter requirement is clearly equivalent to $\pi_{\Gi, \Lambda}$ being monotone for every $\Lambda \Subset V$.
\begin{lem}
Let $\Gi$ be an attractive Gibbs structure, and let $(F_n)$ be any sequence  of finite subsets in $V$  such that $F_i \subset F_j$ for $i<j$ and $\bigcup F_i = V$.  Then limit
\[
\lim_{i \to \infty} \bp_{\Gi, F_i}(\delta_{\omega^+})
\]
exists and equals to the unique maximal Gibbs measure.
Also the limit
\[
\lim_{i \to \infty} \bp_{\Gi, F_i}(\delta_{\omega^-})
\]
exists and equals to the unique minimal Gibbs measure.
\end{lem}
\begin{proof}
It is enough to prove the statement for maximal Gibbs measure.
At first let us prove that if the limit
\[
\lim_{i \to \infty} \bp_{\Gi, F_i}(\delta_{\omega^+})
\]
exists, then it is a Gibbs measure which is bigger than any other one. Let $\nu$ be any Gibbs measure.
We have that 
\[
\bp_{\Gi, F_i}(\delta_{\omega^+}) \succeq \bp_{\Gi, F_i}(\nu) = \nu.
\]
It is now easy to see that the limit will be bigger than $\nu$. It clearly implies now that any two converging such a sequences will have the same limit. Hence any such sequence will converge, since otherwise we can refine two converging subseqences with different limits. 
\end{proof}
Let $\Gi$ be an attractive Gibbs structure. For any $S \subset V$ we will define the map $\pmax_{\Gi,V}: \Omega \to \Me(\Omega)$, $\omega \mapsto \max\lbrace {\MG_{V\setminus S,\omega}\rbrace}$ and $\pmin_{\Gi,V}: \Omega \to \Me(\Omega)$, $\omega \mapsto \min \lbrace{\MG_{V\setminus S,\omega}\rbrace}$. We can consider also correspondent kernels $\bpmax_{\Gi,V}$ and $\bpmin_{\Gi,V}$. We note that $\Gi_{V \setminus S,\omega}$ has unique Gibbs measure iff $\pmax_{\Gi,V}(\omega) = \pmin_{\Gi,V}(\omega)$.

The following lemma is from \cite{Go80}.
\begin{lem}
If $\nu^{+}$ is a maximal Gibbs measure for attractive Gibbs structure $\Gi$ then for every $S \subset V$ we have $\bpmax_{\Gi,S}(\nu^{+})=\nu^{+}$. Analogously, for minimal Gibbs measure $\nu^{-}$ we have $\bpmin_{\Gi,S}(\nu^{-})=\nu^{-}$.  
\end{lem}
\begin{proof}
It is enough to prove the statement for the maximal Gibbs measure. Let $(F_n)$ be a nested sequence of finite subsets of $V$ whose union equals to $V$. 
It is easy to note that the sequence of kernels $\bp_n: \Me(\Omega) \to \Me(\Omega)$, $\bp_n: \nu \mapsto \bp_{\Gi, S \cap F_n}(\pr_{V \setminus S}(\nu) \otimes \pr_{V}(\delta_{\omega^{+}}))$ converges pointwise to $\bpmax_{\Gi,S}$.
We can clearly see now that $\nu^{+} \preceq \bpmax_{\Gi,S}(\nu^{+})$. We also note that $\pi_{\Gi,F_n}(\delta_{\omega^{+}}) \succeq \pi_n(\nu^{+})$, which follows from the fact that $\bp_{\Gi,F_n}(\delta_{\omega^{+}}) \succeq \nu^{+}$ and that $\bp_n(\bp_{\Gi,F_n}(\delta_{\omega^{+}}))=\bp_{\Gi,F_n}(\delta_{\omega^{+}})$. Passing now to the limit we get $\nu^{+} \succeq \bpmax_{\Gi,S}(\nu^{+})$. This implies that $\nu^{+} = \bpmax_{\Gi,S}(\nu^{+})$.
\end{proof}

\begin{lem}\label{attractive good}
If $\nu$ is a unique Gibbs measure for attractive Gibbs strucure then for any $S \subset V$ and for $\nu$-a.e. $\omega \in \Omega$ we have that $\Gi_{S,\omega}$ has unique Gibbs measure.
\end{lem}
\begin{proof}
From the prewious lemma we have that for $\nu$ --- almost every $\omega \in \Omega$ holds $\nu\vert_{\omega}^{\sB(V )}=\bpmax_{\Gi,V \setminus S}(\omega) = \bpmin_{\Gi,V \setminus S}(\omega)$.
\end{proof}

\section{Gibbs structures and measures over sofic groups}\label{gibbs over sofic}
Let us consider shift-invariant Gibbs structure over the group $G$. Vertex set would be $G$, alphabet would be the same for every $g \in G$: $A$.
We will assume that there is a function $\varphi: A^D \to \R$($D \Subset G$) such that $\Phi_T(\tau) = \varphi(g(\tau))$, $\omega \in A^G$ if $T=Dg$ for some $g \in G$. It is not hard to see that any shift-invariant Gibbs structure could be reduced to this case without changing the set of Gibbs measures. So let us fix such a function $\varphi$ we will call it a potential. From now on we will denote this Gibbs structure $\Gi$ and we will assume that it posesses unique Gibbs measure. The latter will be abbreviated as $\varphi$ is a UGM potential.

In order to use Seward's bound we need to establish freeness of our actions:

\begin{lem}\label{gibbs free}
%Пусть $\nu$ --- инвариантная гиббовская мера для инвариантной гиббсовской структуры $\Gi$ над счётной группой $G$. Тогда получаемое сдвиговое действие является существенно свободным. 
Let $\nu$ be shift-invariant Gibbs measure for shift-invariant Gibbs structure $\Gi$ over the countable group $G$. Then correspondent shift action is essentially free.
\end{lem}

\begin{proof}
Let $g \in G \setminus \lbrace e \rbrace$ be an arbitrary element. It is enough to prove that $g \omega \neq \omega$ for $\nu$-a.e. $\omega \in \Omega$. Consider such a sequence $(g_i)_{i \in \N}$ that all the sets $\lbrace g_i g , g_i\rbrace$ are disjoint. We will define $B_i = \lbrace \omega \in \Omega \vert \omega_{g_i g} = \omega_{g_i} \rbrace$. From the definition of Gibbs measure it follows that there is a constant $c<1$ such that $$\nu(B_0 \cap \ldots \cap B_n \cap B_{n+1}) \leqslant c \cdot \nu (B_0 \cap \ldots \cap B_n)$$ for $n \in \N$ (we can take $c$ to be $\sup_{\omega \in \Omega}(\pi_{\Gi, \lbrace g, e\rbrace}(\omega))(B_0)$). It implies that $\nu(\bigcap_{i \in \N}B_i)$ = 0 which finishes the proof.
\end{proof}

%We also will need ergodicity in order to use equality between sofic entropy and horizontal sofic entropy.

%\begin{lem}\label{gibbs ergodic}
%Unique Gibbs measure for any potential is ergodic.
%\end{lem}
%\begin{proof}
%This is a standard fact. It holds since tail subalgebra of unique Gibbs measure is trivial (follows from the theorem 8.8 in \cite{G11}) and that subalgebra of invariant sets is contained in tail subalgebra (proposition 14.9 there).
%\end{proof}

For sofic group $G$ and potential $\varphi$ we can define a very special sequence of Gibbs structures. For $\Gi^i$ a vertex set would be $V_i$ from the definition of sofic approximation. For every vertex an alphabet set would be the same: $A$.

We will define the potential for $\Gi^i$ in the following way: $\Phi_T(\tau)=\sum_{v \in V_i}\varphi(\theta_v(\tau))$ for $\tau \in A^{V_i}$ and $T \subset V_i$ there the sum is taken over all such $v \in V_i$ that $\sigma_i^{D}(v)=T$(actually, in most instances this sum will contain either one or zero summands). Let us denote $\eta_i$ the unique Gibbs measure for Gibbs structure $\Gi^i$. 
 
In my work \cite{A15} it was proven that the sofic entropy can be computed by means of these measures.
 
\begin{theor}
Suppose $\Gi$ is a shift-invariant Gibbs structure over the sofic group $G$ with fixed sofic approximation, $\nu$ is the unique Gibbs measure for $\Gi$ and sequence $\eta_i$ is defined as above. Then sofic entropy of shift-actions on $A^G$ endowed with the measure $\nu$ can be computed by 
\[
h(\nu)=\limsup_{i \to \infty} \frac{H(\eta_i)}{\lvert V_i \rvert}.
\]
\end{theor} 

\begin{lem}
For any $\Lambda \Subset G$ there is such an $S \Subset G$ that if $v \in V_i$ is an $S$-good element then $\theta_v(\eta_i)$ is $(\Gi,\Lambda)$-admissible.
\end{lem}
\begin{proof}
It is not hard to see that for $S = \partial_{Gi}\Lambda \cup \Lambda \cup D \cup D^{-1}$ the statement will hold. 
\end{proof}

\begin{lem}
Let $\nu$ be a unique Gibbs measure for $\Gi$. 
Let $K$ be a finite subset of $G$. Let us fix any metric for the weak* topology on $\Me(A^K)$. 
For any $\varepsilon>0$ we can find such a finite subset $S$ of $G$ that $K \cup \partial_{\Gi}K \subset S$ and that if $v \in V_i$ is $S$-good point then $\pr_{K}(\theta_v(\eta_i))$ is $\varepsilon$-close to $\pr_{K}(\nu)$.
\end{lem}
\begin{proof}
By the lemma \ref{approximation of gibbs measure} we can obtain such a finite subset $\Lambda \Subset V$, that for every $W \subset V$ with $\Lambda \cup \partial_{\Gi} \Lambda \subset W$ and for every $(\Lambda,\Gi)$-admissible measure $\mu$ on $A^W$ we have that $\pr_{K}(\mu)$ would be $\varepsilon$-close to $\pr_K(\nu)$. Let us now use the previous lemma in order to get such a subset $S$ that for every $S$-good point $v \in V_i$ measure $\theta_v(\eta_i)$ would be $(\Lambda, \Gi)$-admissibsle.
\end{proof}

\begin{lem}
For any $F \Subset G$ there is such an $S \Subset G$ that if $v \in V_i$ is an $S$-good element, then $\partial_{\Gi^i}(\sigma_i^{F}(v)) = \sigma_i^{\partial_{\Gi}F}(v)$.
\end{lem}
\begin{proof}
Take $S= F \cup \partial_{\Gi}F \cup D \cup D^{-1}$
\end{proof}
\begin{lem}
Let $F \Subset G$. Let us fix any metric for the weak-$\star$ topology on $\Me(A^{F \cup \partial_{\Gi}F})$. Let us fix $\varepsilon>0$. Then there is such a finite subset $S$ of $G$ that $S \subset F \cup \partial_{\Gi}F \cup D \cup D^{-1}$ and that for every $S$-good point $v \in V_i$ we have $\partial_{\Gi^i}(\sigma_i^{F}(v)) = \sigma_i^{\partial_{\Gi}F}(v)$ and $\pr_{F \cup \partial_{\Gi}F}(\theta(\eta_i))$ is $\varepsilon$-close to $\pr_{F \cup \partial_{\Gi}F}(\nu)$. 
\end{lem}
\begin{proof}
It is not hard to see that two previous lemma will continue to hold after enlarging of correspondent subsets $S$. So we take $S$ for this lemma to be their union.
\end{proof}

The proof of the following lemma uses the random ordering techique similar to that from \cite{BCKL13}.

\begin{lem}
Suppose $\nu$ be unique Gibbs measure for $\Gi$. Then for any finite set $F \subset G$ we have $h(\nu) \geq \expect_{\xi}H(\alpha\vert \alpha^{L_{\xi}\cup (G \setminus F)})$. 
\end{lem}
\begin{proof}
Let $(\chi_{v})_{v \in V_i}$ be a collection of i.i.d. random variables each of which has the uniform distribution on the unit interval. Let $L_{v,\chi}$ for $v \in V_i$ be the set of all such $w \in V_i$ that $\chi_{w}<\chi_{v}$. Let $(Y_{v})_{v \in V_i}$ be a random process with the distribution $\eta_i$. Let $(X_g)_{g \in G}$ be a random process with the distribution $\nu$.
It is easy to see that by the chain rule expansion
\[
H(\eta_i)=\sum_{v \in V_i}\expect_{\chi}H \left( Y_v \vert (Y_w)_{w \in L_{v,\chi}} \right).
\]
Let us fix any metric for the weak* topology on $\Me(A^{F \cup \partial_{\Gi}F})$ and let $\varepsilon>0$ (we will choose it later). We apply previous lemma to get suitable $S$. Let $V_i' \subset V_i$, $i \in \N$ be sets of all the $S$-good points. We know that $$\lim_{i \to \infty} \lvert V_i'\rvert / \lvert V_i \rvert = 1.$$ Consider now the term in the sum above, corresponding to some element $v \in V_i'$:
\begin{multline*}
\expect_{\chi}H \left( Y_v \vert (Y_w)_{w \in L_{v,\chi}} \right) \geq \expect_{\chi}H \left( Y_v \vert (Y_w)_{w \in L_{v,\chi} \cup (V_i \setminus \sigma_i^{F}(v))} \right) =\\= \expect_{\chi}H \left( Y_v \vert (Y_w)_{w \in (L_{v,\chi} \cap \sigma_i^F(v) )\cup (\partial_{\Gi^i} \sigma_i^{F}(v))} \right) =\\= \expect_{\chi} \left( H((X_w)_{w \in \lbrace v \rbrace \cup (L_{v,\chi} \cap \sigma_i^F(v)) \cup (\partial_{\Gi^i} \sigma_i^{F}(v))}) - H((X_w)_{w \in (L_{v,\chi} \cap \sigma_i^F(v)) \cup (\partial_{\Gi^i} \sigma_i^{F}(v))}) \right) \geq \\ \geq \expect_{\xi}\left( H\left((X_g)_{g \in \lbrace e\rbrace \cup (L_{\xi} \cap F) \cup \partial_{\Gi}F } \right) - H\left((X_g)_{g \in (L_{\xi} \cap F) \cup \partial_{\Gi}F } \right)\right) - \varepsilon'.
\end{multline*}
In estimates above we used Markov property and then the fact for any ${\varepsilon'>0}$ we can choose such a small $\varepsilon>0$ that distribution of $(Y_w)_{w \in \sigma_i^{F \cup \partial_{\Gi}F}}$ will be close enough to that of $(X_g)_{g \in F \cup \partial_{\Gi}F} $(with relabeling by $g \mapsto \sigma_i^{g}(v)$) that for any $R \subset F \cup \partial_{\Gi}F$ we will have
\[
\left\lvert H((X_g)_{g \in R}) - H((Y_w)_{w \in \sigma_i^{R}(v)})\right\rvert < \varepsilon'/2.
\]
Then we again use the Markov property:
\begin{multline*}
\expect_{\xi}\left( H\left((X_g)_{g \in \lbrace e\rbrace \cup (L_{\xi} \cap F) \cup \partial_{\Gi}F } \right) - H\left((X_g)_{g \in (L_{\xi} \cap F) \cup \partial_{\Gi}F } \right)\right) =\\= \expect_{\xi}\left( H\left((X_g)_{g \in \lbrace e\rbrace \cup L_{\xi}  \cup (G \setminus F) } \right) - H\left((X_g)_{g \in L_{\xi} \cup (G \setminus F) } \right)\right)=\\=\expect_{\xi}H\left(X_e \vert (X_g)_{g \in L_{\xi} \cup (G \setminus F)}\right) = \expect_{\xi}H(\alpha\vert \alpha^{L_{\xi}\cup (G \setminus F)})
\end{multline*}
The estimate $\lim_{i \to \infty} \lvert V_i'\rvert / \lvert V_i\rvert=1$ (lemma \ref{sofic is good}) implies now that
\[
h(\nu) \geq \limsup_{i \to \infty} \frac{H(\eta_i)}{\lvert V_i \rvert} \geq \expect_{\xi}H(\alpha\vert \alpha^{L_{\xi}\cup (G \setminus F)}) - \varepsilon'.
\]
This finishes the proof, since $\varepsilon'$ could be taken arbitrarily small.
\end{proof}

\begin{proof}[Proof of theorem \ref{main ineq}]
Take $F_n$ to be a sequence of co-finite subsets of $G$ such that $F_i \subset F_j$ for $i>j$ and $\bigcap_{i \in \N}F_i = \varnothing$.
For every $\xi$ we have $$H(\alpha \vert \alpha^{L_{\xi}\cup F_n}) \to H(\alpha \vert \widetilde{\alpha^{L_{\xi}}}).$$ So the theorem holds by the previous lemma and the dominated convergence theorem.
\end{proof}


\begin{thebibliography}{100000}



\bibitem[A15]{A15} A. Alpeev, 
\textit{The entropy of Gibbs measures on sofic groups} Zap. Nauchn. Sem. POMI, 436 (2015): 34--48.

\bibitem[A16]{A16} A. V. Alpeev, \textit{Announce of an entropy formula for a class of Gibbs measures}, Zap. Nauchn. Sem. POMI, 448, (2016): 7--13

\bibitem[AS17]{AS17} A. Alpeev and B. Seward, 
\textit{Krieger’s finite generator theorem for actions of countable groups III}, in preparation.

\bibitem[AuP17]{AuP17} T. Austin and M. Podder, \textit{Gibbs measures over locally tree-like graphs and percolative entropy over infinite regular trees}, arXiv preprint arXiv:1705.03589 (2017).

\bibitem[BCKL13]{BCKL13}C. Borgs, J. Chayes, J. Kahn, and L. Lovász, \textit{Left and right convergence of graphs with bounded degree}. Random Structures and Algorithms, 42.1(2013): 1--28.

\bibitem[B10a]{B10a} L. Bowen, \textit{A measure conjugacy invariant for actions of free groups}, Annals of Mathe-
matics 171 (2010), no. 2, 1387--1400.

\bibitem[B10b]{B10b} L. Bowen, \textit{Measure conjugacy invariants for actions of countable sofic groups}, Journal of the American Mathematical Society 23 (2010): 217--245.

\bibitem[B10c]{B10c} L. Bowen, \textit{The ergodic theory of free group actions: entropy and the f-invariant}, Groups,
Geometry, and Dynamics 4 (2010), no. 3, 419–432.

\bibitem[B10d]{B10d} L. Bowen, \textit{Non-abelian free group actions: Markov processes, the Abramov–Rohlin formula and Yuzvinskii’s formula}, Ergodic Theory and Dynamical Systems 30.06 (2010): 1629--1663.

\bibitem[B11]{B11} L. Bowen, \textit{Entropy for expansive algebraic actions of residually finite groups}, Ergodic Theory and Dynamical Systems 31.03 (2011): 703--718.

\bibitem[B17]{B17} L. Bowen, \textit{Examples in the entropy theory of countable group actions}, arXiv preprint arXiv:1704.06349 (2017).

\bibitem[BL12]{BL12} L. Bowen and H. Li, \textit{Harmonic models and spanning forests of residually finite groups}, Journal of Functional Analysis 263.7 (2012): 1769--1808.



\bibitem[C15]{C15} A. Carderi, \textit{Ultraproducts, weak equivalence and sofic entropy}, arXiv preprint arXiv:1509.03189 (2015).

\bibitem[DM10a]{DM10a} A. Dembo and A. Montanari, \textit {Gibbs measures and phase transitions on sparse random graphs}, Brazilian Journal of Probability and Statistics (2010): 137-211.

\bibitem[DM10b]{DM10b} A. Dembo and A. Montanari, \textit{Ising models on locally tree-like graphs}, The Annals of Applied Probability 20.2 (2010): 565-592.

\bibitem[D68]{D68} R. L. Dobrushin, \textit{Description of a random field by means of conditional probabilities and conditions for its regularity}, Teor. Verojatnost. i Primenen 13 (1968): 201–-229.

\bibitem[EW11]{EW11} M. Einsiedler and T. Ward,  \textit{Ergodic theory with a view towards number theory}. Graduate texts in mathematics, 259. Springer, London, 2011.

\bibitem[FS15]{FS15} F. Rassoul-Agha and T. Seppäläinen, \textit{A course on Large Deviations with an Introduction to Gibbs Measures}, Vol. 162. American Mathematical Soc., 2015.

\bibitem[GS15]{GS15}
D. Gaboriau and B. Seward,
\textit{Cost, $l^2$-Betti numbers, and the sofic entropy of some algebraic actions}, arXiv preprint arXiv:1509.02482 (2015).

\bibitem[G11]{G11} H.-O. Georgii, \textit{Gibbs measures and phase transitions}. Vol. 9. Walter de Gruyter, 2011.



\bibitem[Go80]{Go80} S. Goldstein, \textit{Remarks on the global Markov property}, Communications in Mathematical Physics 74.3 (1980): 223-234.

\bibitem[H14]{H14} B. Hayes, \textit{Fuglede-kadison determinants and sofic entropy}, arXiv preprint arXiv:1402.1135 (2014).

%\bibitem[Ka42]{Ka42} L. V. Kantorovich, \textit{On the translocation of masses}, Dokl. Akad. Nauk SSSR, 37, Nos. 7–8(1942): 227–-229

%\bibitem[K95]{K95}
%A. S. Kechris, 
%\textit{Classical Descriptive Set Theory}. Springer-Verlag, New York, 1995.

\bibitem[Ke13]{Ke13} D. Kerr, \textit{Sofic measure entropy via finite partitions}, Groups Geom. Dyn, 7(2013): 617--632.

\bibitem[OW87]{OW87} D. S. Ornstein and B. Weiss, \textit{Entropy and isomorphism theorems for actions of amenable groups}, Journal d'Analyse Mathématique 48.1 (1987): 1-141.

\bibitem[R13]{R13} U.A. Rozikov, \textit{Gibbs Measures on Cayley Trees}, World Scientific, Singapore (2013).

\bibitem[S14a]{S14a}
B. Seward,
\textit{Krieger's finite generator theorem for actions of countable groups I}, preprint. http://arxiv.org/abs/1405.3604.

\bibitem[S14a]{S14b}
B. Seward,
\textit{Krieger's finite generator theorem for actions of countable groups II}, preprint.

\bibitem[S16]{S16} 
B. Seward, \textit{Weak containment and Rokhlin entropy} arXiv preprint arXiv:1602.06680 (2016).

\end{thebibliography}
\end{document}